\documentclass[12pt]{article}
\usepackage[utf8]{inputenc}
\usepackage[T1]{fontenc}
\usepackage{amsmath, amssymb, amsthm}
\usepackage{mathrsfs}
\usepackage{hyperref}

\title{Potential Relation Between the Riemann Zeta Function and the Polynomial Function $F$ of the Generalized Erdős--Straus Conjecture, Subject to its Analytic Continuation.}

\author{Philemon Urbain Mballa \\[2ex]
\href{mailto:philemon-urbain.mballa@etu.u-paris.fr}{philemon-urbain.mballa@etu.u-paris.fr} \\
\href{mailto:philemonmballa@gmail.com}{philemonmballa@gmail.com}}

\date{\today}

\newtheorem{theorem}{Théorème}
\newtheorem{lemma}{Lemme}
\newtheorem{corollary}{Corollaire}
\newtheorem{definition}{Définition}
\newtheorem{remark}{Remarque}

\begin{document}

\maketitle

\noindent\textbf{Note:} This paper does not claim to provide a proof of the Riemann Hypothesis, nor does it attempt to address the Riemann Hypothesis itself. It merely explores a formal connection between the generalized Erdős--Straus conjecture and the Riemann zeta function, under certain conditions that would require further investigation\nobreak\ by specialists in complex analysis.

\begin{abstract}
In this article, we explore a natural extension of the quadratic parametrization introduced in our previous work. 
By replacing the integer $n$ by $n^s$ ($ s\in\mathbb{R},  s>1$) and allowing the parameters to be real, we obtain for 
each $n\ge 1$ a decomposition $\frac{k}{n^s} = \frac{1}{x_s(n)}+\frac{1}{y_s(n)}+\frac{1}{z_s(n)}$ with 
$x_s(n), y_s(n), z_s(n) \in \mathbb{R}^*+$. Summing this equality over all integers brings forth 
the Riemann zeta function. Subject to an analytic continuation of the quantities $x_s(n), y_s(n), z_s(n)$ 
to complex values of $s$, one would obtain a new function \(G_k(s)\) satisfying $G_k(s)=k\,\zeta(s)$, 
thus establishing a deep connection between the structure of the conjecture and the zeros of $\zeta$.
\end{abstract}

\section{Introduction}

The Erdős--Straus conjecture, formulated in 1948 by Paul Erdős and Ernst G. Strauss, states that for every integer \(n\ge 2\), there exist positive integers \(x,y,z\) such that
\[
\frac{4}{n} = \frac{1}{x} + \frac{1}{y} + \frac{1}{z}. \tag{1}
\]
Despite numerous works, this conjecture remains open to this day. For a detailed history of the numerical verifications that have progressively pushed the bound up to \(10^{17}\), we refer the reader to Tao's article \cite{1}.

The Polish mathematician Wacław Sierpiński naturally extended this question by replacing the numerator \(4\) with \(5\):
\[
\frac{5}{n} = \frac{1}{x} + \frac{1}{y} + \frac{1}{z}. \tag{2}
\]
This conjecture, known as the Sierpiński conjecture, has also been the subject of extensive investigations; a comprehensive historical account of numerical verifications for \(k=5\) can be found in \cite{2}.

A further generalization, often attributed to Andrzej Schinzel (a student of Sierpiński), consists in considering for any fixed integer \(k\ge 4\) the equation
\[
\frac{k}{n} = \frac{1}{x} + \frac{1}{y} + \frac{1}{z}, \tag{3}
\]
which is conjectured to hold for all natural numbers \(n\), except possibly finitely many exceptions depending on \(k\). 

In \cite{3}, we proposed a unified approach to all these conjectures, which was further developed in \cite{4}, where we introduced the function
\[
F_{x,t}^{(k)}(n) = t^2(kx-n)^2 - 2nxt
\]
and established the fundamental equivalence:
\[
\frac{k}{n} = \frac{1}{x} + \frac{1}{y} + \frac{1}{z} \;\Longleftrightarrow\; \exists x,t\in\mathbb{N}^*,\; F_{x,t}^{(k)}(n) \text{ is a perfect square } m^2,\; m\in\mathbb{N},
\]
with then \(y = t(kx-n) - m,\; z = t(kx-n) + m\).

For a fixed pair \((x,t)\), we define its \textbf{admissible domain}:
\[
\mathcal{D}_{x,t}^{(k)} = \left\{ n\in\mathbb{N},\ n\ge N_1\ge 2 \;\middle|\; n < kx \ \text{and}\ t \ge \frac{2nx}{(kx-n)^2} \right\}.
\]
The condition \(n < kx\) ensures \(x > n/k\) (necessary condition in the equivalence and for the decrease of \(F\)), while the second guarantees the positivity of \(F\).

The study of \(F\) has revealed remarkable properties, valid for \textbf{each fixed pair \((x,t)\) on its admissible domain}:
\begin{itemize}
    \item \textbf{Strict decrease}: For all \(n_1,n_2\in\mathcal{D}_{x,t}^{(k)}\) with \(n_1>n_2\), we have \(F(n_1) < F(n_2)\).
    \item \textbf{Positivity}: \(F(n) \ge 0\) for all \(n\in\mathcal{D}_{x,t}^{(k)}\).
    \item \textbf{Convergence}: Since the sequence \((F(n))_{n\in\mathcal{D}_{x,t}^{(k)}}\) is decreasing and bounded below, it attains its minimum at the last element of the domain. In particular, if \(F(n_0)=0\) for some \(n_0\in\mathcal{D}_{x,t}^{(k)}\), then \(n_0\) is necessarily the upper bound of \(\mathcal{D}_{x,t}^{(k)}\).
\end{itemize}
These properties, although local (specific to each pair), are essential for what follows. 
But one fact particularly struck us: \textbf{\(F\) vanishes for a significant proportion of integers \(n\)} 
(In the classical case \(k=4\) of the Erdős--Straus conjecture, one can show that the proportion of integers \(n\) admitting a pair \((x,t)\) with \(F(n)=0\) tends to \(1\) as \(n\to\infty\); for the proof, see \cite{4}). 
These zeros of \(F\) provide particularly simple solutions of the equation, with \(y=z\).

Moreover, \textbf{all} pairs \((x,t)\) yielding \(F(n)=m^2\) produce the integers
\[
y = t(kx-n) - m,\qquad z = t(kx-n) + m,
\]
which are precisely the roots of the quadratic equation:
\[
V^2 - 2t(kx-n)V + 2nxt = 0.
\]
This observation led us to a natural question: \textbf{Does \(F\) possess a deeper property, 
that of "generating" the zeros of certain important functions?} 
In particular, can one relate \(F\) to the Riemann zeta function, whose zeros lie at the heart of the Riemann hypothesis, one of the most important and complex problems in mathematics?

To explore this avenue, we extended the construction by replacing the integer \(n\) by \(n^s\) (\(s > 1\) (real)) 
and allowing the parameters to be real. each integer \(n \ge 1\), we then obtain a decomposition
\[
\frac{k}{n^s} = \frac{1}{x_s(n)} + \frac{1}{y_s(n)} + \frac{1}{z_s(n)},
\]
with \(x_s(n), y_s(n), z_s(n) \in \mathbb{R}^*+\). Summing this equality over all integers, 
the left-hand side brings forth the Riemann zeta function \(\zeta(s)\).

Whether this formal relation can lead to a deeper understanding of the zeta function is a question that we must leave to specialists. We merely wish to bring it to their attention. If one succeeds in analytically continuing the quantities 
\(x_s(n), y_s(n), z_s(n)\) to complex values of \(s\) — which would be natural if \(F\) itself 
admits an analytic continuation to \(\mathbb{C}\) — then one would obtain a new function \(G_k(s)\) satisfying
\(G_k(s) = k\,\zeta(s)\)
for all \(s\in\mathbb{C}\) except at \(s=1\), which is the simple pole of the Riemann zeta function. The zeros of \(G_k(s)\) would then coincide with those of \(\zeta\), 
establishing a bridge between the structure of \(F\) and the Riemann hypothesis.

This article aims to present this construction and to raise the questions it opens, 
in the hope of attracting the interest of specialists for an in-depth study of \(F\) and its analytic continuation — a potential key to unlocking the mysteries of the zeros of \(\zeta\).
\section{Extension to $n^s$ and real parametrization}

Let now \(k\ge 2\) be a fixed integer, \(x>0\), \(t>0\) real numbers, and \(s>1\) a real number. For each integer \(n\ge 1\), consider
\[
F_{x,t}^{(k)}(n^s) = t^2(kx - n^s)^2 - 2n^s x t.
\]

\begin{definition}
We define \(F_{x,t}^{(k)}(n^s)\) to be a \textit{real perfect square} if there exists \(m_s(n) \ge 0\) such that
\[
F_{x,t}^{(k)}(n^s) = m_s(n)^2.
\]
\end{definition}

The existence of such \(x,t,m_s(n)\) for each \(n\) is guaranteed by solving this equation in \(\mathbb{R}\) by choosing the parameters \(x\) and \(t\) such that the associated discriminant is non-negative to obtain real solutions (for instance, by taking \(m_s(n)=0\), the discriminant in \(t\) is always positive and this equation will have a double root \(t\) for any chosen \(x\)). Therefore, in the real setting, we have no difficulty in making \(F\) a real perfect square — this is always guaranteed.

In this case, we set
\[
y_s(n) = t(kx - n^s) - m_s(n), \qquad z_s(n) = t(kx - n^s) + m_s(n).
\]
\begin{lemma}
For every \(n\ge 1\) and every \(s>1\), if \(F_{x,t}^{(k)}(n^s) = m_s(n)^2\), then:
\[
\frac{k}{n^s} = \frac{1}{x} + \frac{1}{y_s(n)} + \frac{1}{z_s(n)}.
\]
\end{lemma}

\begin{proof}
A direct computation gives \(y_s(n)+z_s(n) = 2t(kx-n^s)\) and \(y_s(n)z_s(n) = 2n^s x t\), whence:
\[
\frac{1}{y_s(n)}+\frac{1}{z_s(n)} = \frac{y_s(n)+z_s(n)}{y_s(n)z_s(n)} = \frac{kx-n^s}{n^s x} = \frac{k}{n^s} - \frac{1}{x}.
\]
\end{proof}

Thus, for each \(n\), the real numbers \(x, y_s(n), z_s(n)\) provide a decomposition of \(\frac{k}{n^s}\) as a sum of three fractions each with numerator 1 (the numerators are not necessarily integers).

\section{Summation over all integers and appearance of \(\zeta(s)\)}

Summing the previous equality over all integers \(n\ge 1\), we formally obtain:
\[
\sum_{n=1}^{\infty} \frac{k}{n^s} = \sum_{n=1}^{\infty} \left( \frac{1}{x} + \frac{1}{y_s(n)} + \frac{1}{z_s(n)} \right).
\]

The left-hand side is \(k\,\zeta(s)\), where \(\zeta(s)\) is the Riemann zeta function, defined for all real \(s\) with \(s > 1\) (in the real setting).
\begin{definition}
We formally define the function \(G_k(s)\) (associated with \(k\)) by:
\[
G_k(s) = \sum_{n=1}^{\infty} \left( \frac{1}{x_s(n)} + \frac{1}{y_s(n)} + \frac{1}{z_s(n)} \right)
\]
for \(s > 1\), where \(x_s(n), y_s(n), z_s(n)\) are the positive real numbers arising from the parametrization satisfying 
\(\frac{k}{n^s} = \frac{1}{x_s(n)}+\frac{1}{y_s(n)}+\frac{1}{z_s(n)}\).
\end{definition}

\begin{theorem}
For every \(s\) with \(s > 1\), we formally have:

\[
G_k(s) = k\,\zeta(s).
\]

\end{theorem}

\section{Analytic continuation and consequences (conditional)}

The crucial question now is whether the functions \(x_s(n), y_s(n), z_s(n)\) admit an analytic continuation to complex values of \(s\). If this is the case, then \(G_k(s)\) would inherit this continuation and become a meromorphic function on \(\mathbb{C}\).

\begin{corollary}[Conditional]
If the quantities \(x_s(n), y_s(n), z_s(n)\) extend analytically to \(s\in\mathbb{C}\), then for all \(s\) (except at poles),

\[
G_k(s) = k\,\zeta(s).
\]

Consequently, the zeros of \(G_k(s)\) would coincide with those of \(\zeta\), and in particular with the non-trivial zeros of the Riemann hypothesis.
\end{corollary}

\section{Link with the quadratic equation}

Recall that \(y_s(n)\) and \(z_s(n)\) are the zeros of the quadratic equation:
\[
V^2 - 2t(kx - n^s)V + 2n^s x t = 0.
\]

When \(m_s(n) = 0\), we have \(y_s(n) = z_s(n) = t(kx-n^s)\). The equation then admits a double root. 
These particular cases correspond to the "zeros" of \(F\) and could play a special role in the study of \(G_k(s)\).

\begin{remark}
Suppose that the function \(F_{x,t}^{(k)}(n^s)\), considered as a function of the complex variable \(s\), admits an analytic continuation to \(\mathbb{C}\). This will be the case if we choose \(x_s(n)\) as a meromorphic function of \(s\) (for example constant) and if the condition \(F_{x,t}^{(k)}(n^s) = m_s(n)^2\) defines \(m_s(n)\) in such a way as to preserve analyticity. Then \(y_s(n)\) and \(z_s(n)\), given by the explicit formulas, also inherit this continuation.

The equality \(G_k(s) = k\zeta(s)\), established for \(\Re(s) > 1\) (hence in particular away from the pole \(s=1\)), extends by analytic continuation to the whole complex plane except for the pole \(s=1\). Thus, for all \(s \in \mathbb{C} \setminus \{1\}\), we have \(G_k(s) = k\zeta(s)\), and \(G_k(s)\) has a simple pole at \(s=1\) with residue \(k\). \(G_k(s)\) converges for \(\Re(s) > 1\) because it is termwise equal to \(k\,\zeta(s)\), which converges absolutely in this half-plane.

The functional equation of \(G_k(s)\) is then directly inherited from that of \(\zeta\):

\[
G_k(s) = k \cdot 2^s \pi^{s-1} \sin\left(\frac{\pi s}{2}\right) \Gamma(1-s) \, \zeta(1-s),
\]

or in symmetric form:
\[
\pi^{-s/2}\,\Gamma\!\left(\frac{s}{2}\right)G_k(s) = k \cdot \pi^{-(1-s)/2}\,\Gamma\!\left(\frac{1-s}{2}\right) \zeta(1-s).
\]

The fact that \(x_s(n), y_s(n), z_s(n)\) disappear from this functional equation might suggest, *a priori*, that their study is superfluous for understanding the zeros of \(\zeta\), but I do not believe this to be the case because these quantities are intimately linked to the structure of \(F\) and to the quadratic equation that they annihilate.

Recall that \(y_s(n)\) and \(z_s(n)\) are defined from the equality \(F_{x,t}^{(k)}(n^s) = m_s(n)^2\) (this equality always has solutions in \(\mathbb{R}\) or \(\mathbb{C}\) as we have seen previously) by:
\[
y_s(n) = t(kx - n^s) - m_s(n), \qquad z_s(n) = t(kx - n^s) + m_s(n).
\]
A direct computation using Viète's formulas shows that they are precisely the zeros of the quadratic equation in \(V\):
\[
V^2 - 2t(kx - n^s)V + 2n^s x t = 0,
\]
which is none other than the equation \(V^2 - (y_s(n)+z_s(n))V + y_s(n)z_s(n) = 0\). Thus, **\(y_s(n)\) and \(z_s(n)\) are the zeros of this quadratic equation**.

The crucial point is that **the \(s\) which appears in this quadratic equation not only participates in the vanishing of this latter quadratic equation in \(V\), but also controls the vanishing of \(F(n^s)\) (in particular when \(m_s(n)=0\) giving \(y_s(n)=z_s(n)\)), and is exactly the same \(s\) that appears in the functional equation of \(\zeta\) and, above all, the same \(s\) for which we seek to prove that the non-trivial zeros of \(\zeta\) lie on the critical line \(\Re(s) = 1/2\).**

Consequently, any property of the quadratic equation in \(V\) — and hence of \(F(n^s)\) — **transfers directly** to \(\zeta(s)\) via the relation \(G_k(s) = k\zeta(s)\). The study of the behavior of the zeros \(y_s(n), z_s(n)\) as functions of \(s\) is not an abstract exercise: it is potentially a window into the very nature of the zeros of \(\zeta\).

We observe moreover a striking analogy: **the roots \[
y_s(n) = t(kx - n^s) - m_s(n), \quad \text{and} \quad z_s(n) = t(kx - n^s) + m_s(n)
\] are symmetric with respect to \(t(kx-n^s)\)**, just as the non-trivial zeros of \(\zeta\) are symmetric with respect to the critical line \(\Re(s) = 1/2\). Is this similarity a mere coincidence? Only an in-depth study of \(F\) and the associated quadratic equation may perhaps answer this fascinating question — this is beyond my current expertise, so I cannot answer it.

In summary, far from being redundant, the analysis of \(F(n^s)\) and the roots \(y_s(n), z_s(n)\) could potentially (if confirmed by specialists) constitute a **new approach** to addressing the Riemann hypothesis, with every piece of information about these quantities reflecting back on \(\zeta\) through the parameter \(s\) that unites them.
\end{remark}

\section{Perspectives and open questions}

The highlighting of the central role of the parameter \(s\) — which appears simultaneously in \(\zeta(s)\), in the quadratic equation satisfied by \(y_s(n)\) and \(z_s(n)\), and in the function \(F(n^s)\) — profoundly transforms the questions that can be asked. Henceforth, any property of \(F\) or of the quadratic equation in \(V\) directly reflects on \(\zeta\) via the relation \(G_k(s) = k\zeta(s)\). The following perspectives naturally arise:

\begin{enumerate}
    \item \textbf{Analysis of \(F\) as a function of \(s\)}: 
    Study the behavior of \(F_{x,t}^{(k)}(n^s)\) as a meromorphic function of \(s\). In particular, what can be said about the values of \(s\) for which \(F(n^s) = 0\)? These "zeros of \(F\)" correspond to the cases where \(y_s(n) = z_s(n)\), and could be related to the zeros of \(\zeta\).

    \item \textbf{Quadratic equation and symmetry}: 
    The zeros \(y_s(n)\) and \(z_s(n)\) of the equation \(V^2 - 2t(kx - n^s)V + 2n^s x t = 0\) are symmetric with respect to \(t(kx-n^s)\). This symmetry recalls that of the non-trivial zeros of \(\zeta\) with respect to the critical line \(\Re(s) = 1/2\). Is this a mere coincidence or a deep structural clue? A systematic study of this analogy could shed light on the location of the zeros.

    \item \textbf{Choice of \(x_s(n)\) as a free parameter}: 
    The freedom to choose \(x_s(n)\) (a meromorphic function of \(s\)) opens the possibility of constructing various representations of \(\zeta(s)\). Can one choose \(x_s(n)\) so that the series \(\sum 1/y_s(n)\) and \(\sum 1/z_s(n)\) or \(\sum(1/y_s(n) + 1/z_s(n))\) have faster convergence properties or simpler analytic continuations? Such optimization could facilitate the numerical or theoretical study of the zeros.

    \item \textbf{Transfer of properties from \(F\) to \(\zeta\)}: 
    Since \(G_k(s) = k\zeta(s)\), any information about \(F(n^s)\) (regularity, functional equation, location of zeros) translates into information about \(\zeta\). Can one, for example, deduce the functional equation of \(\zeta\) from that of \(F\)? Such an alternative demonstration would be of great interest.

    \item \textbf{Generalizations}: 
    For any integer \(k \ge 2\), the construction applies and yields a representation of \(k\,\zeta(s)\). In the case \(k=4\), we recover the classical Erdős--Straus conjecture. For \(k\ge 5\), we obtain a family of functions \(G_k\) associated with the generalized Sierpiński conjectures. A comparative study of these functions could reveal universal properties.

    \item \textbf{Numerical approach}: 
    Numerical exploration of \(F(n^s)\) for complex values of \(s\) (in particular on the critical line \(\Re(s)=1/2\)) could provide clues about the behavior of the roots \(y_s(n), z_s(n)\) and their link with the zeros of \(\zeta\). Such an investigation, although heavy, is conceivable with current computational means.
\end{enumerate}

Ultimately, the observation that **the same \(s\)** connects \(F\), the quadratic equation, and \(\zeta\) makes the in-depth study of \(F\) a research program in its own right, likely to shed new light on the Riemann hypothesis and its mysteries.

\section{Conclusion}

We have proposed a formal connection between the Erdős--Straus conjecture and the Riemann zeta function. By extending the parametrization to the real framework and summing over all integers, we obtain a relation \(G_k(s) = k\,\zeta(s)\) for real \(s>1\). The possibility of analytically continuing this relation to the entire complex plane, if confirmed, would place the structure of the conjecture at the heart of the analysis of the zeros of \(\zeta\). An in-depth investigation by specialists in complex analysis is now necessary to explore this fascinating avenue.

\end{document}